\documentclass[a4paper,12pt,reqno]{amsart}

\usepackage[T1]{fontenc}
\usepackage[utf8]{inputenc}
\usepackage{lmodern}
\usepackage[english]{babel}

\usepackage[%
	left=2.5cm,       
	right=2.5cm,      
	top=3.5cm,        
	bottom=3.5cm,     
	heightrounded,    
	bindingoffset=0mm 
]{geometry}

\usepackage{amsmath}
\usepackage{amssymb}
\usepackage{mathtools}
\usepackage{mathrsfs}

\numberwithin{equation}{section}

\usepackage{hyperref} 
\usepackage[noabbrev,capitalize]{cleveref}

\usepackage{bookmark}

\theoremstyle{plain}
	\newtheorem{theorem}{Theorem}[section]
	\newtheorem{lemma}[theorem]{Lemma}
	\newtheorem{proposition}[theorem]{Proposition}
	\newtheorem{corollary}[theorem]{Corollary}
	
\theoremstyle{definition}

	\newtheorem{remark}[theorem]{Remark}
	
\usepackage[initials]{amsrefs}

\usepackage{xcolor}

\usepackage{enumerate}
	
\newcommand{\N}{\mathbb{N}}
\newcommand{\R}{\mathbb{R}}
\newcommand{\eps}{\varepsilon}
\newcommand{\m}{\mathfrak{m}}
\newcommand{\de}{\partial}
\newcommand{\fl}{\mathcal{L}}

\newcommand{\E}{\mathcal{E}}

\DeclareMathOperator{\di}{d\!}
\DeclareMathOperator{\sgn}{sgn}

\renewcommand{\phi}{\varphi}
\renewcommand{\rho}{\varrho}
\renewcommand{\d}{\mathsf{d}}
					                                       
\DeclarePairedDelimiter{\set}{\{}{\}}

\DeclarePairedDelimiter{\abs}{|}{|}

\mathchardef\ordinarycolon\mathcode`\:
\mathcode`\:=\string"8000
\begingroup \catcode`\:=\active
  \gdef:{\mathrel{\mathop\ordinarycolon}}
\endgroup

\allowdisplaybreaks

\begin{document}

\title[Existence and uniqueness theorems for some semi-linear equations]{Existence and uniqueness theorems for some semi-linear equations on locally finite graphs}

\author[A.~Pinamonti]{Andrea Pinamonti}
\address[A.~Pinamonti]{Università degli Studi di Trento, Dipartimento di Matematica, via Sommarive, 14, 38123 Povo (Trento), Italy}
\email[A.~Pinamonti]{andrea.pinamonti@unitn.it}

\author[G.~Stefani]{Giorgio Stefani}
\address[G.~Stefani]{Department Mathematik und Informatik, Universit\"at Basel, Spiegelgasse 1, CH-4051 Basel, Switzerland}
\email{giorgio.stefani@unibas.ch}

\date{\today}



\thanks{
\textit{Acknowledgements}.
The authors are members of INdAM--GNAMPA. 
The first author is partially supported by the INdAM--GNAMPA Project 2020 \textit{Convergenze variazionali per funzionali e operatori dipendenti da campi vettoriali}.
The second author is partially supported by the ERC Starting Grant 676675 FLIRT -- \textit{Fluid Flows and Irregular Transport}, by the INdAM--GNAMPA Project 2020 \textit{Problemi isoperimetrici con anisotropie} (n.\ prot.\ U-UFMBAZ-2020-000798 15-04-2020), by the INdAM--GNAMPA 2022 Project \textit{Analisi geometrica in strutture subriemanniane}, codice CUP\_E55\-F22\-000\-270\-001, and has received funding from the European Research Council (ERC) under the European Union’s Horizon 2020 research and innovation program (grant agreement No. 945655).
}

\begin{abstract}
We study some semi-linear equations for the $(m,p)$-Laplacian operator on locally finite weighted graphs.
We prove existence of weak solutions for all $m\in\N$ and $p\in(1,+\infty)$ via a variational method already known in the literature by exploiting the continuity properties of the energy functionals involved.   
When $m=1$, we also establish a uniqueness result in the spirit of the Brezis--Strauss Theorem.
We finally provide some applications of our main results by dealing with some Yamabe-type and Kazdan--Warner-type equations on locally finite weighted graphs.
\end{abstract}

\maketitle

\section{Introduction}

\subsection{Framework}

When dealing with PDEs coming from the Euler--Lagrange equations of some energy functional, existence and multiplicity results of weak solutions are usually achieved via the so-called Variational Method.
 
In the recent years, this approach has been employed by many authors in order to deal with a large variety of interesting PDEs on graphs, see~\cites{CGY97,CGY00,CCP19,CG98,G18,G20,GHJ18,GJ18-1,GJ18-2,GJ19,Gri01,Gri18,GLY16-K,GLY16-Y,GLY17,HSZ20,LW17-2,LY20,Z17,ZZ18,ZC18,ZL18,ZL19} and the references therein.
Of particular interest for the scopes of the present paper is the work~\cite{GLY16-Y}, where the authors proved existence of weak solutions for a Yamabe-type equation on locally finite weighted graphs via the celebrated Mountain Pass Theorem due to Ambrosetti and Rabinowitz~\cite{AR73}. 

The main aim of this note is twofold. 
On one hand, by exploiting some ideas developed in~\cites{FMBR16,MBR17} in the context of Carnot groups we prove the existence of weak solutions for a Yamabe-type equation on locally finite weighted graphs.
Our result is similar to the one of~\cite{GLY16-Y} but holds under a different set of assumptions. 
On the other hand, we adapt the strategy of~\cite{BMP07} developed in the Euclidean setting to establish a uniqueness result for the weak solutions of Yamabe-type equations on locally finite weighted graphs in the spirit of the celebrated Brezis--Strauss Theorem~\cite{BS73}.

\subsection{Main notation}

Before stating our main results, we need to recall some notation, see \cref{dfe} for the precise definitions.

Given $G=(V,E)$ a locally finite non-oriented graph, the \emph{vertex boundary} $\de\Omega$
and the \emph{vertex interior} $\Omega^\circ$ of a connected subgraph $\Omega\subset V$ are defined as
\begin{align*}
\de\Omega=\set*{x\in\Omega : \exists y\notin\Omega\  \text{such that}\ xy\in E},
\qquad
\Omega^\circ=\Omega\setminus\de\Omega.
\end{align*}
We say that $\Omega$ is \emph{bounded} if it is a bounded subset of~$V$ with respect to the usual \emph{vertex distance} $\d\colon V\times V\to[0,+\infty)$. 

Once a symmetric \emph{weight function} $w\colon V\times V\to [0,\infty)$ is given, we can define the \emph{Laplacian} of a function $u\colon V\to\R$ as
\begin{equation}\label{eq:laplacian}
\Delta u(x)=\frac{1}{\m(x)}\sum_{y\in V}w_{xy}(u(y)-u(x))
\quad
\text{for}\ x\in V,
\end{equation}
where $\m\colon V\to[0,+\infty)$ is the measure function
\begin{equation}\label{defW}
\m(x)=\sum_{y\in V}w_{xy}
\quad
\text{for all}\ x\in V.
\end{equation}
The \emph{gradient form} associated to the Laplacian operator is the bilinear symmetric form
\begin{equation*}
\Gamma(u,v)(x)
=
\frac{1}{2\m(x)}\sum_{y\in V}w_{xy}(u(y)-u(x))(v(y)-v(x)),
\quad
x\in V,
\end{equation*}
defined for any couple of functions $u,v\colon V\to\R$.
As a consequence, the \emph{slope} of the function $u\colon V\to\R$ is given by
\begin{equation*}
|\nabla u|(x)
=
\sqrt{\Gamma(u,u)(x)}
=
\left(\frac{1}{2\m(x)}\sum_{y\in V}w_{xy}(u(y)-u(x))^2\right)^\frac{1}{2}
\quad
\text{for}\ x\in V.
\end{equation*} 
Note that 
$|\Gamma(u,v)|\le |\nabla u|\,|\nabla v|$
for any couple of functions $u,v\colon V\to\R$.
In analogy with the Euclidean framework, for any $m\in\N$ we recursively define the $m$-\emph{slope} of the function~$u$ as
\begin{equation*}
|\nabla^m u|=
\begin{cases}
|\nabla(\Delta^\frac{m-1}{2} u)| & \text{if $m$ is odd},\\[3mm]
|\Delta^\frac{m}{2} u| & \text{if $m$ is even},
\end{cases}
\end{equation*} 
where $|\Delta^\frac{m}{2} u|$ denotes the usual absolute value of the function $\Delta^\frac{m}{2} u$.   
The natural operator associated to the Sobolev spaces $(W^{m,p}_0(\Omega),\|\cdot\|_{W^{m,p}_0(\Omega)})$ (see~\eqref{eq:sobolev_norm} below for the precise definition) is the \emph{$(m,p)$-Laplacian operator}
\begin{equation*}
\fl_{m,p}\colon W^{m,p}_0(\Omega)\to L^p(\Omega)
\end{equation*}
defined in the distributional sense for all $u\in W^{m,p}_0(\Omega)$ as
\begin{equation}
\label{eq:def_mp_Lap}
\int_\Omega \fl_{m,p}u\,\phi\di\m=
\begin{cases}
\displaystyle\int_\Omega |\nabla^m u|^{p-2}\,\Gamma(\Delta^{\frac{m-1}{2}}u,\Delta^{\frac{m-1}{2}}\phi)\di\m 
& \text{if $m$ is odd},\\[5mm]
\displaystyle\int_\Omega |\nabla^m u|^{p-2}\,\Delta^{\frac{m}{2}}u\,\Delta^{\frac{m}{2}}\phi\di\m 
& \text{if $m$ is even},
\end{cases}
\end{equation} 
whenever $\phi\in W^{m,p}_0(\Omega)$. 

The $(m,p)$-Laplacian $\fl_{m,p} u$ can be explicitly computed at any point of $\Omega$. 
In particular, $\fl_{1,p}$ is the \emph{$p$-Laplacian operator}, given by
\begin{equation}
\label{eq:p-laplacian}
\Delta_p u(x)
=
\frac{1}{\m(x)}\sum_{y\in\Omega}\left(|\nabla u|^{p-2}(y)+|\nabla u|^{p-2}(x)\right)w_{xy}(u(y)-u(x)),
\quad
x\in\Omega,
\end{equation}
for all $u\in W^{1,p}_0(\Omega)$.
When $p=2$, we recover the usual Laplacian operator defined in~\eqref{eq:laplacian}. 

\subsection{Main results}
We are now ready to state our main results.
Our first main theorem is the following existence  result for a Yamabe-type equation for the $(m,p)$-Laplacian operator on locally finite weighted graphs. 

\begin{theorem}\label{th:main}
Let $G=(V,E)$ be a weighted locally finite graph. 
Let $\Omega\subset V$ be a bounded domain such that $\Omega^\circ\ne\varnothing$ and $\de\Omega\ne\varnothing$. 
Let $m\in\N$, $p\in(1,+\infty)$ and $q\in[p-1,+\infty)$. 
Let $f\colon\Omega\times\R\to\R$ be a Carathéodory function such that
\begin{equation}\label{eq:growth}
|f(x,t)|\le a(x)+b(x)\,|t|^q 
\quad 
\text{for every } 
(x,t)\in\Omega\times\R
\end{equation}
for some non-negative $a,b\in L^1(\Omega)$ with $\|a\|_{L^1(\Omega)},\|b\|_{L^1(\Omega)}>0$. 
There exists 
\begin{equation}\label{eq:Lambda}
\Lambda
=
\Lambda(m,p,q,\|a\|_{L^1(\Omega)},\|b\|_{L^1(\Omega)})
>0
\end{equation}
such that the Yamabe-type problem
\begin{equation}\label{eq:yamabe}
\begin{cases}
\fl_{m,p} u=\lambda f(x,u) & \text{in } \Omega^\circ\\[2mm]
|\nabla^j u|=0 & \text{on } \de\Omega, \quad 0\le j\le m-1,
\end{cases}
\end{equation}
admits at least one non-trivial solution $u_\lambda\in W^{m,p}_0(\Omega)$ for every $0<\lambda<\Lambda$.
\end{theorem}

We observe that the growth condition of the function $f$ assumed in~\eqref{eq:growth} of \cref{th:main} is different from the one assumed in~\cite{GLY16-Y}*{Theorem~3}.
In particular, we do not assume that $f(x,0)=0$ for all $x\in\Omega$.
We also underline that the existence threshold~\eqref{eq:Lambda} depends uniquely on the growth of the function~$f$ and not on the first eigenvalue of the $(m,p)$-Laplacian, as instead it happens in~\cite{GLY16-Y}*{Theorem~3}.

Our second main result is the following uniqueness theorem for a Yamabe-type equation for the $p$-Laplacian operator on locally finite weighted graphs in the spirit of the famous Brezis--Strauss Theorem, see~\cites{BS73,BMP07}.

\begin{theorem}\label{res:main_uniqueness}
Let $G=(V,E)$ be a weighted locally finite graph. 
Let $\Omega\subset V$ be a bounded domain such that $\Omega^\circ\ne\varnothing$ and $\de\Omega\ne\varnothing$. 
Let $p\in[1,+\infty)$ and let $g\colon\Omega\times\R\to\R$ be a function such that $g(x,0)=0$ and $t\mapsto g(x,t)$ is non-decreasing for all $x\in\Omega$.
If $f_1,f_2\in L^1(\Omega)$, $h\in L^1(\de\Omega)$ and $u_1,u_2\in W^{1,p}(\Omega)$ solve the problems
\begin{equation*}
\begin{cases}
-\Delta_p u_i + g(x,u_i) = f_i & \text{in}\ \Omega^\circ\\[2mm]
u_i = h & \text{on}\ \de\Omega
\end{cases}
\qquad
\text{for}\ i=1,2,
\end{equation*}
then
\begin{equation}
\label{eq:oscillation}
\int_\Omega |g(x,u_1)-g(x,u_2)|\di\m
\le
\int_\Omega |f_1-f_2|\di\m.
\end{equation}
As a consequence, for every $f\in L^1(\Omega)$ and $h\in L^1(\de\Omega)$ the problem
\begin{equation}\label{eq:uniqueness_main}
\begin{cases}
-\Delta_p u + g(x,u) = f & \text{in}\ \Omega^\circ\\[2mm]
u = h & \text{on}\ \de\Omega
\end{cases}
\end{equation}	
admits at most one solution $u\in W^{1,p}(\Omega)$.
\end{theorem}

By combining \cref{th:main} and \cref{res:main_uniqueness}, we get the following well-posedness result for a Yamabe-type problem for the $p$-Laplacian on locally finite weighted graphs. 

\begin{proposition}\label{propcomb}
Let $G=(V,E)$ be a weighted locally finite graph. 
Let $\Omega\subset V$ be a bounded domain such that $\Omega^\circ\ne\varnothing$ and $\de\Omega\ne\varnothing$. 
Let $p\in(1,+\infty)$, $q\in[p-1,+\infty)$ and $a,b\in L^1(\Omega)$ with $\inf_\Omega b\ge0$.
There exists 
\begin{equation*}
\Lambda
=
\Lambda(m,p,q,\|a\|_{L^1(\Omega)},\|b\|_{L^1(\Omega)})
>0
\end{equation*}
such that the Yamabe-type problem
\begin{equation}\label{eq:yamabe_cor}
\begin{cases}
-\Delta_p u + b|u|^{q-1}u = a
& \text{in}\ \Omega^\circ\\[2mm]
u=0
& \text{on}\ \de\Omega
\end{cases}
\end{equation}
has a unique solution $u\in W^{1,p}_0(\Omega)$.
\end{proposition}

\subsection{Organization of the paper}

The structure of the paper is the following. 
In \cref{dfe} we recall the preliminary definitions and notions needed in the paper. 
Sections~\ref{proof1} and~\ref{proof2} are devoted to the proofs of Theorems~\ref{th:main} and~\ref{res:main_uniqueness} respectively.
Finally, in \cref{proof3} we provide some applications of our main results, along with the proof of \cref{propcomb}.

\section{Preliminaries}\label{dfe}

In this section, we introduce the main notation and some preliminary results we will need in the sequel of the paper.

\subsection{Non-oriented graphs}
Let $V$ be a non-empty set and let $E\subset V\times V$. 
We write
\begin{equation*}
x\sim y
\iff
xy=(x,y)\in E.
\end{equation*}
We will always assume that 
\begin{equation*}
xy\in E \iff yx\in E.
\end{equation*}
We say that the couple $G=(V,E)$ is a \emph{non-oriented graph} with \emph{vertices} $V$ and \emph{edges} $E$.

The non-oriented graph $G$ is \emph{locally finite} if
\begin{equation*}
\#\set*{y\in V : xy\in E}<+\infty
\quad
\text{for all}\ 
x\in V,
\end{equation*}
that is, each vertex in $V$ belongs to a finite number of edges in $E$. 

Given $n\in\N$, a \emph{path} on $G$ is any finite sequence of vertices $\set{x_k}_{k=1,\dots,n}\subset V$ such that 
\begin{equation*}
x_kx_{k+1}\in E
\quad
\text{for all}\ k=1,\dots,n-1.
\end{equation*}
The \emph{length} of a path on $G$ is the number of edges in the path.
We say that $G$ is \emph{connected} if, for any two vertices $x,y\in V$, there is a path connecting $x$ and $y$. 
If $G$ is connected, then the function $\mathsf{d}\colon V\times V\to[0,+\infty)$ given by
\begin{equation*}
\mathsf{d}(x,y)=\min\set*{n\in\N_0 :\text{$x$ and $y$ can be connected by a path of length $n$}},
\end{equation*}
for $x,y\in V$, is a distance on $V$.
As a consequence, any connected locally finite non-oriented graph has at most countable many vertices.

Let $G=(V,E)$ be a locally finite non-oriented graph.
A \emph{weight} on $G$ is a function $w\colon V\times V\to[0,+\infty)$, $w(x,y)=w_{xy}$ for $x,y\in V$, such that 
\begin{equation*}
w_{xy}=w_{yx}
\quad
\text{and} 
\quad
w_{xy}>0
\iff
xy\in E
\end{equation*}
for all $x,y\in V$.
We conclude this section by pointing out that the function $\m\colon V\to[0,+\infty)$ defined in \eqref{defW}
can be interpreted as a measure on the graph by simply setting
\begin{equation*}
\int_V u\di\m
=
\int_V u(x)\di\m(x)
=
\sum_{y\in V} u(x)\,\m(x)\in[0,+\infty]
\end{equation*}
for any function $u\colon V\to[0,+\infty)$.

\subsection{Sobolev spaces on bounded domains}
Let $G=(V,E)$ be a weighted locally finite graph and let $\Omega\subset V$ be a bounded domain.
Note that the integral
\begin{equation*}
\int_\Omega u\di\m
=
\int_\Omega u(x)\di\m(x)
=
\sum_{x\in\Omega}u(x)\,\m(x)
\end{equation*}
of a function $u\colon\Omega\to\R$ is well defined, since $\Omega$ is a finite set. 
Let $p\in[1,+\infty]$ and $m\in\N_0$.
The \emph{Sobolev space $W^{m,p}(\Omega)$} is the set of all functions $u\colon\Omega\to\R$ such that
\begin{equation}\label{eq:sobolev_norm}
\|u\|_{W^{m,p}(\Omega)}
=
\sum_{k=0}^m\|\nabla^k u\|_{L^p(\Omega)}<+\infty.
\end{equation}
When $m=0$, this space is simply the \emph{Lebesgue space} $L^p(\Omega)$.
Since $\Omega$ is a finite set, the Banach space $(W^{m,p}(\Omega),\|\cdot\|_{W^{m,p}(\Omega)})$ is finite dimensional and, actually, coincides with the set of all real-valued functions on~$\Omega$.
 
For $m\in\N$, we define
\begin{equation}
\label{eq:def_C_m0}
C^m_0(\Omega):=\set*{u\colon\Omega\to\R : |\nabla^k u|=0\ \text{on}\ \de\Omega\ \text{for all}\ 0\le k\le m-1}
\end{equation}    
and we let $W^{m,p}_0(\Omega)$ be the completion of $C^m_0(\Omega)$ with respect to the Sobolev norm~\eqref{eq:sobolev_norm}. 
The following result is proved in~\cite{GLY16-Y}*{Theorem~7}. 

\begin{theorem}[Sobolev embedding]\label{th:sobolev}
Let $G=(V,E)$ be a locally finite graph and let $\Omega\subset V$ be a bounded domain such that $\Omega^\circ\ne\varnothing$ and $\de\Omega\ne\varnothing$. 
Let $m\in\N$ and $p\in[1,+\infty)$. 
The space $W^{m,p}_0(\Omega)$ is continuously embedded in $L^q(\Omega)$ for all $q\in[1,+\infty]$, i.e.\ there exists a constant $C_{m,p}>0$, depending only on $m$, $p$ and $\Omega$, such that
\begin{equation}\label{eq:sobolev}
\|u\|_{L^q(\Omega)}
\le 
C_{m,p}
\|\nabla^m u\|_{L^p(\Omega)}
\end{equation}
for all $q\in[1,+\infty]$ and $u\in W^{m,p}_0(\Omega)$.
\end{theorem}

By \cref{th:sobolev}, the space $(W^{m,p}_0(\Omega),\|\cdot\|_{W^{m,p}_0(\Omega)})$ is a finite dimensional Banach space, where
\begin{equation}\label{eq:sobolev_0_norm}
\|u\|_{W^{m,p}_0(\Omega)}
=
\|\nabla^m u\|_{L^p(\Omega)}
\end{equation}
is a norm on $W^{m,p}_0(\Omega)$ equivalent to the norm~\eqref{eq:sobolev_norm}. 
Since $\Omega$ is a finite set, the Banach space $(W^{m,p}_0(\Omega),\|\cdot\|_{W^{m,p}_0(\Omega)})$ is finite dimensional and, actually, coincides with the set~$C^m_0(\Omega)$ defined in~\eqref{eq:def_C_m0}.

\section{Proof of \texorpdfstring{\cref{th:main}}{Theorem 1.1}}\label{proof1}

In this section, we prove our first main result following the strategy outlined in~\cite{FMBR16}. Given $\lambda>0$, we define
\begin{equation}\label{eq:def_Phi_Psi}
\Phi(u)
=
\|u\|_{W^{m,p}_0(\Omega)}, 
\qquad 
\Psi_\lambda(u)
=
\lambda\int_\Omega F(x,u)\di\m,
\end{equation}
for all $u\in W^{m,p}_0(\Omega)$, where 
\begin{equation}\label{eq:def_F}
F(x,t)
=
\int_0^t f(x,\tau)\di\tau
\quad
\text{for all}\ t\in\R. 
\end{equation}
Note that, thanks to the assumption in~\eqref{eq:growth}, the functional $\Psi_\lambda$ is well defined and (strongly) continuous on $W^{m,p}_0(\Omega)$.
Indeed, we can estimate
\begin{align*}
|F(x,t)|
\le
\int_0^{|t|}|f(x,\tau)|\di\tau
\le
\int_0^{|t|}a(x)+b(x)\,|\tau|^q\di\tau
=
a(x)\,|t|+b(x)\,\frac{|t|^{1+q}}{1+q}
\end{align*}
for all $(x,t)\in\Omega\times\R$, so that
\begin{equation*}
|\Psi_\lambda(u)|
\le
\lambda\left(
\|a\|_{L^1(\Omega)}\|u\|_{L^\infty(\Omega)}
+
\|b\|_{L^1(\Omega)}\,\frac{\|u\|_{L^\infty(\Omega)}^{1+q}}{1+q}
\right)
\end{equation*}
which is finite for all $u\in W^{m,p}_0(\Omega)$ by \cref{th:sobolev}.
In addition, if $(u_n)_{n\in\N}\subset W^{m,p}_0(\Omega)$ is converging to some $u\in W^{m,p}_0(\Omega)$, then $u_n\to u$ in $L^\infty(\Omega)$ as $n\to+\infty$ and thus
\begin{align*}
\lim_{n\to+\infty}
\Psi_\lambda(u_n)
=
\lim_{n\to+\infty}
\sum_{x\in\Omega} F(x,u_n(x))\,\m(x)
=
\sum_{x\in\Omega} F(x,u(x))\,\m(x)
=
\Psi_\lambda(u)
\end{align*}
by the continuity of the function $t\mapsto F(x,t)$ for $x\in\Omega$ fixed.

The following two results are proved in~\cite{FMBR16}*{Lemma~3.2 and Lemma~3.3} respectively for the case $p=2$. 
Here we reproduce the proofs in our setting in the more general case $p\in(1,+\infty)$ for the reader's ease.

\begin{lemma}\label{lemma:tic}
Let $p\in(1,+\infty)$ and $\lambda>0$. 
If
\begin{equation}
\label{eq:tic_ipotesi}
\limsup_{\eps\to0+}\,
\frac{\sup\limits_{u\in\Phi^{-1}([0,\rho])}\Psi_\lambda(u)-\sup\limits_{u\in\Phi^{-1}([0,\rho-\eps])}\Psi_\lambda(u)}{\eps}
<
\rho^{p-1}
\end{equation}
for some $\rho>0$, then
\begin{equation}\label{eq:inf_rho_sigma}
\inf_{\sigma<\rho}\,
\frac{\sup\limits_{u\in\Phi^{-1}([0,\rho])}\Psi_\lambda(u)-\sup\limits_{u\in\Phi^{-1}([0,\sigma])}\Psi_\lambda(u)}{\rho^p-\sigma^p}
<
\frac{1}{p}.
\end{equation} 
\end{lemma}

\begin{proof}
Let $\eps\in(0,\rho)$ and note that 
\begin{align*}
\lim_{\eps\to0^+}
\frac{\eps}{\rho^p-(\rho-\eps)^p}
=
\frac{1}{p\rho^{p-1}}.
\end{align*}
Therefore, in virtue of~\eqref{eq:tic_ipotesi}, we get that
\begin{align*}
\limsup_{\eps\to0^+}\,
&
\frac{\sup\limits_{u\in\Phi^{-1}([0,\rho])}\Psi_\lambda(u)-\sup\limits_{u\in\Phi^{-1}([0,\rho-\eps])}\Psi_\lambda(u)}{\rho^p-(\rho-\eps)^p}
\\
&=
\limsup_{\eps\to0^+}\,
\frac{\sup\limits_{u\in\Phi^{-1}([0,\rho])}\Psi_\lambda(u)-\sup\limits_{u\in\Phi^{-1}([0,\rho-\eps])}\Psi_\lambda(u)}{\eps}
\cdot
\frac{\eps}{\rho^p-(\rho-\eps)^p}
\\
&=
\frac{1}{p\rho^{p-1}}\,
\limsup_{\eps\to0^+}\,
\frac{\sup\limits_{u\in\Phi^{-1}([0,\rho])}\Psi_\lambda(u)-\sup\limits_{u\in\Phi^{-1}([0,\rho-\eps])}\Psi_\lambda(u)}{\eps}
<
\frac1p.
\end{align*} 
Thus we can find $\bar\eps\in(0,\rho)$ such that
\begin{align*}
\frac{\sup\limits_{u\in\Phi^{-1}([0,\rho])}\Psi_\lambda(u)-\sup\limits_{u\in\Phi^{-1}([0,\rho-\bar\eps])}\Psi_\lambda(u)}{\rho^p-(\rho-\bar\eps)^p}
<
\frac1p	
\end{align*}
and so $\bar\sigma=\rho-\bar\eps<\rho$ gives
\begin{align*}
\inf_{\sigma<\rho}\,
\frac{\sup\limits_{u\in\Phi^{-1}([0,\rho])}\Psi_\lambda(u)-\sup\limits_{u\in\Phi^{-1}([0,\sigma])}\Psi_\lambda(u)}{\rho^p-\sigma^p}
<
\frac{\sup\limits_{u\in\Phi^{-1}([0,\rho])}\Psi_\lambda(u)-\sup\limits_{u\in\Phi^{-1}([0,\bar\sigma])}\Psi_\lambda(u)}{\rho^p-\bar\sigma^p}
<
\frac1p
\end{align*}
proving~\eqref{eq:inf_rho_sigma}. 
The proof is complete. 
\end{proof}

\begin{lemma}\label{lemma:tac}
Let $p\in(1,+\infty)$ and $\lambda>0$. If~\eqref{eq:inf_rho_sigma} holds for some $\rho>0$, then
\begin{equation}
\label{eq:tac_inf}
\inf_{u\in\Phi^{-1}([0,\rho))}
\frac{\sup\limits_{v\in\Phi^{-1}([0,\rho])}\Psi_\lambda(v)-\Psi_\lambda(u)}{\rho^p-\|u\|_{W^{m,p}_0(\Omega)}^p}
<
\frac{1}{p}.
\end{equation} 
\end{lemma}

\begin{proof}
In virtue of~\eqref{eq:inf_rho_sigma}, we can find $\bar\sigma\in(0,\rho)$ such that
\begin{align*}
\sup\limits_{u\in\Phi^{-1}([0,\bar\sigma])}\Psi_\lambda(u)
>
\sup\limits_{u\in\Phi^{-1}([0,\rho])}\Psi_\lambda(u)
-
\frac{1}{p}(\rho^p-\bar\sigma^p).
\end{align*}
Since the functional $\Psi_\lambda$ is continuous on $W^{m,p}_0(\Omega)$, we can find $\bar u\in W^{m,p}_0(\Omega)$ with $\|\bar u\|_{W^{m.p}_0(\Omega)}=\bar\sigma$ such that
\begin{align*}
\sup\limits_{u\in\Phi^{-1}([0,\bar\sigma])}\Psi_\lambda(u)
=
\sup\limits_{\|u\|_{W^{m,p}_0(\Omega)}=\,\bar\sigma}\Psi_\lambda(u)
=
\Psi_\lambda(\bar u)
\end{align*}
and so
\begin{align*}
\Psi_\lambda(\bar u)
>
\sup\limits_{u\in\Phi^{-1}([0,\rho])}\Psi_\lambda(u)
-
\frac{1}{p}(\rho^p-\bar\sigma^p).
\end{align*}
We thus conclude that
\begin{align*}
\inf_{u\in\Phi^{-1}([0,\rho))}
\frac{\sup\limits_{v\in\Phi^{-1}([0,\rho])}\Psi_\lambda(v)-\Psi_\lambda(u)}{\rho^p-\|u\|_{W^{m,p}_0(\Omega)}^p}
<
\frac{\sup\limits_{v\in\Phi^{-1}([0,\rho])}\Psi_\lambda(v)-\Psi_\lambda(\bar u)}{\rho^p-\|\bar u\|_{W^{m,p}_0(\Omega)}^p}
<
\frac1p
\end{align*}
proving~\eqref{eq:tac_inf}.
The proof is complete. 
\end{proof}

We are now ready to prove our first main result, in analogy with~\cite{FMBR16}*{Theorem~3.1}.

\begin{proof}[Proof of \cref{th:main}]
Let $\lambda>0$ and consider the energy functional $\E_\lambda\colon W^{m,p}_0(\Omega)\to\R$ defined as
\begin{align*}
\E_\lambda(u)
=
\frac{\Phi(u)^p}{p}-\Psi_\lambda(u) 
\quad
\text{for all}\ u\in W^{m,p}_0(\Omega),
\end{align*}
where $\Phi$ and $\Psi_\lambda$ are as in~\eqref{eq:def_Phi_Psi}. 
By the growth condition~\eqref{eq:growth} and \cref{th:sobolev}, we have that $\E_\lambda\in C^1(W^{m,p}_0(\Omega);\R)$, with derivative at $u\in W^{m,p}_0(\Omega)$ given by
\begin{align*}
\E'_\lambda(u)[\phi]
=
\begin{cases}
\displaystyle\int_\Omega |\nabla^m u|^{p-2}\,\Gamma(\Delta^{\frac{m-1}{2}}u,\Delta^{\frac{m-1}{2}}\phi)\di\m
-
\lambda\int_\Omega f(x,u)\,\phi\di\m
& 
\text{if $m$ is odd},\\[5mm]
\displaystyle\int_\Omega |\nabla^m u|^{p-2}\,\Delta^{\frac{m}{2}}u\,\Delta^{\frac{m}{2}}\phi\di\m
-
\lambda\int_\Omega f(x,u)\,\phi\di\m
&
\text{if $m$ is even},
\end{cases}
\end{align*}
for any $\phi\in W^{m,p}_0(\Omega)$. 
In particular, the solutions of the problem~\eqref{eq:yamabe} are exactly the critical points of the functional $\E_\lambda$. 
Now let $\rho>0$ to be fixed later. 
Since $\E_\lambda$ is a continuous functional on~$W^{m,p}_0(\Omega)$, there exists $u_{\lambda,\rho}\in\Phi^{-1}([0,\rho])$ such that
\begin{equation}\label{eq:minimality_u}
\E_\lambda(u_{\lambda,\rho})
=
\inf_{u\in\Phi^{-1}([0,\rho])}\E_\lambda(u). 
\end{equation}
To conclude the proof, we just need to show that $\|u_{\lambda,\rho}\|_{W^{m,p}_0(\Omega)}<\rho$. 
To this aim, for $\eps\in(0,\rho)$ we consider
\begin{equation*}
\Lambda(\rho,\eps)
=
\frac{\sup\limits_{u\in\Phi^{-1}([0,\rho])}\Psi_\lambda(u)-\sup\limits_{u\in\Phi^{-1}([0,\rho-\eps])}\Psi_\lambda(u)}{\eps}.
\end{equation*}
Recalling the definition of $\Psi_\lambda$ in~\eqref{eq:def_Phi_Psi}, we have
\begin{align*}
\Lambda(\rho,\eps)
&=
\frac{1}{\eps}\,
\bigg(\sup\limits_{u\in\Phi^{-1}([0,\rho])}\Psi_\lambda(u)-\sup\limits_{u\in\Phi^{-1}([0,\rho-\eps])}\Psi_\lambda(u)\bigg)\\
&\le
\frac{\lambda}{\eps}\,
\sup_{u\in\Phi^{-1}([0,1])}\,
\int_\Omega\,\abs*{\int_{(\rho-\eps)u(x)}^{\rho u(x)}|f(x,t)|\di t\,}\di\m(x).
\end{align*}
Thanks to the growth condition~\eqref{eq:growth}, we can estimate
\begin{align*}
\frac{\lambda}{\eps}\int_\Omega\bigg|\int_{(\rho-\eps)u(x)}^{\rho u(x)}&|f(x,t)|\di t\,\bigg|\di\m(x)\\
&\le
\frac{\lambda}{\eps}\int_\Omega \eps\,a(x)|u(x)|
+
b(x)\left(\frac{\rho^{q+1}-(\rho-\eps)^{q+1}}{q+1}\right)|u(x)|^{q+1}\di\m(x)\\
&\le
\lambda\|u\|_{L^\infty(\Omega)}\|a\|_{L^1(\Omega)}+\frac{\lambda\|u\|_{L^\infty(\Omega)}^{q+1}\|b\|_{L^1(\Omega)}}{q+1}\left(\frac{\rho^{q+1}-(\rho-\eps)^{q+1}}{\eps}\right)
\end{align*}
for all $u\in W^{m,p}_0(\Omega)$.
Thus, by the embedding inequality~\eqref{eq:sobolev}, we get
\begin{equation*}
\Lambda(\rho,\eps)
\le
\lambda\,C_{m,p}\|a\|_{L^1(\Omega)}
+
\frac{\lambda\, C_{m,p}^{q+1}\|b\|_{L^1(\Omega)}}{q+1}\left(\frac{\rho^{q+1}-(\rho-\eps)^{q+1}}{\eps}\right)
\end{equation*}
and so
\begin{equation*}
\limsup_{\eps\to0+}
\Lambda(\rho,\eps)
\le
\lambda
\left(
C_{m,p}\|a\|_{L^1(\Omega)}
+
C_{m,p}^{q+1}\|b\|_{L^1(\Omega)}\,\rho^q
\right).
\end{equation*}
We now define
\begin{equation}\label{eq:precise_Lambda}
\lambda_\rho
=
\frac{\rho^{p-1}}{C_{m,p}\|a\|_{L^1(\Omega)}+C_{m,p}^{q+1}\|b\|_{L^1(\Omega)}\,\rho^q}
\in
(0,+\infty)
\end{equation}
and, consequently,
\begin{equation*}
\Lambda
=
\sup_{\rho>0}\lambda_\rho\in(0,+\infty)
\end{equation*}
(note that $\Lambda<+\infty$ is ensured by the fact that $q\ge p-1$).
Now fix $\lambda<\Lambda$ and choose the parameter $\rho>0$ in such a way that $\lambda<\lambda_\rho<\Lambda$. 
This choice implies that
\begin{align*}
\limsup_{\eps\to0+}
\Lambda(\rho,\eps)
&\le
\lambda
\left(
C_{m,p}\|a\|_{L^1(\Omega)}
+
C_{m,p}^{q+1}\|b\|_{L^1(\Omega)}\,\rho^q
\right)
\\
&<
\lambda_\rho
\left(
C_{m,p}\|a\|_{L^1(\Omega)}
+
C_{m,p}^{q+1}\|b\|_{L^1(\Omega)}\,\rho^q
\right)
<\rho^{p-1},
\end{align*}
so that
\begin{align*}
\limsup_{\eps\to0+}\,
\frac{\sup\limits_{u\in\Phi^{-1}([0,\rho])}\Psi_\lambda(u)-\sup\limits_{u\in\Phi^{-1}([0,\rho-\eps])}\Psi_\lambda(u)}{\eps}
<
\rho^{p-1}.
\end{align*}
We can now apply \cref{lemma:tic} to get that
\begin{equation*}
\inf_{\sigma<\rho}\frac{\sup\limits_{u\in\Phi^{-1}([0,\rho])}\Psi_\lambda(u)-\sup\limits_{u\in\Phi^{-1}([0,\sigma])}\Psi_\lambda(u)}{\rho^p-\sigma^p}<\frac{1}{p}
\end{equation*}
and so, by \cref{lemma:tac}, we infer that
\begin{equation*}
\inf_{u\in\Phi^{-1}([0,\rho))}\frac{\sup\limits_{v\in\Phi^{-1}([0,\rho])}\Psi_\lambda(v)-\Psi_\lambda(u)}{\rho^p-\|u\|_{W^{m,p}_0(\Omega)}^p}<\frac{1}{p}.
\end{equation*}
The above inequality implies that there exists $w_{\lambda,\rho}\in\Phi^{-1}([0,\rho))$ such that
\begin{equation*}
\sup_{v\in\Phi^{-1}([0,\rho])}\Psi_\lambda(v)
<
\Psi_\lambda(w_{\lambda,\rho})
+
\frac{\rho^p-\|w_{\lambda,\rho}\|_{W^{m,p}_0(\Omega)}^p}{p}.
\end{equation*}
Now, if by contradiction we assume that $\|u_{\lambda,\rho}\|_{W^{m,p}_0(\Omega)}=\rho$, then the previous inequality implies that
\begin{equation*}
\Psi_\lambda(u_{\lambda,\rho})
<
\Psi_\lambda(w_{\lambda,\rho})
+
\frac{\rho^p-\|w_{\lambda,\rho}\|_{W^{m,p}_0(\Omega)}^p}{p}
\end{equation*}
which is equivalent to $\E_\lambda(u_{\lambda,\rho})>\E_\lambda(w_{\lambda,\rho})$, contradicting~\eqref{eq:minimality_u}.
The proof is complete.
\end{proof}

\begin{remark}[The precise value of $\Lambda$ in \cref{th:main}]\label{rem:Lambda}
Note that the above proof allows to give a precise value to the existence threshold $\Lambda>0$ in \cref{th:main}.
Indeed, one just need to find the maximal value of the function defined in~\eqref{eq:precise_Lambda}, which is explicitly computable in term of $p$, $q$, $\|a\|_{L^1(\Omega)}$, $\|b\|_{L^1(\Omega)}$ and $C_{m,p}$.
In particular, in the limiting case $q=p-1$, one has
\begin{align*}
\Lambda
=
\lim_{\rho\to+\infty}
\frac{\rho^{p-1}}{C_{m,p}\|a\|_{L^1(\Omega)}+C_{m,p}^{p}\|b\|_{L^1(\Omega)}\,\rho^{p-1}}
=
\frac{1}{C_{m,p}^{p}\|b\|_{L^1(\Omega)}},
\end{align*}
which does not depend on $\|a\|_{L^1(\Omega)}$.
\end{remark}

\section{Proof of \texorpdfstring{\cref{res:main_uniqueness}}{Theorem 1.2}}\label{proof2}

In this section we prove \cref{res:main_uniqueness}. 
The overall strategy is to adapt the line developed in~\cite{BMP07}*{Appendix~B} for the Euclidean setting to the present framework. 
Note that~\cite{BMP07} is focused on the case $p=2$ only.
Nonetheless, exploiting the explicit expression~\eqref{eq:p-laplacian} of the $p$-Laplacian, we are able to extend the approach of~\cite{BMP07} also to the case $p\ne2$.

We begin with the following result, analogous to~\cite{BMP07}*{Lemma~B.1}. 

\begin{lemma}\label{res:H_function}
Let $G=(V,E)$ be a weighted locally finite graph. 
Let $\Omega\subset V$ be a bounded domain such that $\Omega^\circ\ne\varnothing$ and $\de\Omega\ne\varnothing$. 
Let $p\in[1,+\infty)$ and $f\in L^1(\Omega)$. 
If $u\in W^{1,p}_0(\Omega)$ is a solution of the problem
\begin{equation}
\label{eq:p_Lap_H_function}	
\begin{cases}
-\Delta_p u = f & \text{in}\ \Omega^\circ\\[1mm]
u = 0 & \text{on}\ \de\Omega,
\end{cases}
\end{equation}
then
\begin{equation*}
\int_\Omega f\, H(u)\di\m \ge0
\end{equation*}
for every non-decreasing locally Lipschitz function $H\colon\R\to\R$ such that $H(0)=0$.
\end{lemma}

\begin{proof}
We start by observing that $H(u)\in W^{1,p}_0(\Omega)$. 
Indeed, $H(v)\in C^0_0(\Omega)$ for all $v\in C^0_0(\Omega)$ with $|\nabla H(v)|\le L|\nabla v|$ on~$\Omega$, where $L=\mathrm{Lip}(H,[-c,c])$, $c=\|v\|_{L^\infty(\Omega)}$.
Using $H(u)$ as a test function in~\eqref{eq:p_Lap_H_function}, we get
\begin{align*}
\int_\Omega f\,H(u)\di\m
=
-\int_\Omega\Delta_p u\, H(u)\di\m
=
\int_\Omega|\nabla u|^{p-2}\,\Gamma(u,H(u))\di\m\ge0,
\end{align*}
because
\begin{align*}
\Gamma(u,H(u))(x)
=
\frac{1}{2\m(x)}\sum_{y\in\Omega}w_{xy}(u(y)-u(x))(H(u(y))-H(u(x)))\ge0
\end{align*}
since $H$ is non-decreasing.
The proof is complete.
\end{proof}

As a consequence, and in analogy with~\cite{BMP07}*{Proposition~B.2}, from \cref{res:H_function} we deduce the following result.

\begin{corollary}\label{res:sgn}
Let $G=(V,E)$ be a weighted locally finite graph. 
Let $\Omega\subset V$ be a bounded domain such that $\Omega^\circ\ne\varnothing$ and $\de\Omega\ne\varnothing$. 
Let $p\in[1,+\infty)$, $M>0$ and $f\in L^1(\Omega)$. 
If $u\in W^{1,p}_0(\Omega)$ is a solution of the problem
\begin{equation*}
\begin{cases}
-\Delta_p u = f & \text{in}\ \Omega^\circ\\[1mm]
u = 0 & \text{on}\ \de\Omega,
\end{cases}
\end{equation*}
then
\begin{equation*}
\int_{\Omega\cap\set*{u\ge M}} f\di\m \ge0,
\quad
\int_{\Omega\cap\set*{u\le -M}} f\di\m \ge0.
\end{equation*}
In particular,
\begin{equation*}
\int_{\Omega\cap\set*{|u|\ge M}} f\,\sgn(u)\di\m \ge0,
\end{equation*}
where $\sgn\colon\R\to\R$ is the \emph{sign function} defined by $\sgn(t)=\frac{t}{|t|}$ for $t\ne0$ and $\sgn(0)=0$.
\end{corollary}

\begin{proof}
For every $n\in\N$ such that $n>\frac1M$, we let $H_n\colon\R\to\R$ be the function
\begin{equation*}
H_n(t)
=
\begin{cases}
0 & \text{for}\ t\le M-\frac1n\\[2mm]
nt-nM+1 & \text{for}\ M-\frac1n<t<M\\[2mm]
1 & \text{for}\ t\ge M.
\end{cases}
\end{equation*}
Since $H_n$ is Lipschitz, non-decreasing and such that $H_n(0)=0$, by \cref{res:H_function} we get that
\begin{equation*}
\int_\Omega f\,H_n(u)\di\m\ge0.
\end{equation*}
Passing to the limit as $n\to+\infty$, we find that
\begin{equation*}
\int_{\Omega\cap\set*{u\ge M}} f\di\m\ge0,
\end{equation*}
as desired.
The conclusion thus follows by linearity.  
\end{proof}

We are now ready to prove our second main result, in analogy with~\cite{BMP07}*{Corollary~B.1}.

\begin{proof}[Proof of \cref{res:main_uniqueness}]
The function $v=u_1-u_2\in W^{1,p}_0(\Omega)$ solves the problem
\begin{equation}\label{eq:v_problem}
\begin{cases}
-\Delta_p v = F & \text{in}\ \Omega^\circ\\[2mm]
v = 0 & \text{on}\ \de\Omega
\end{cases}
\end{equation}
with $F=f_1-f_2-g(x,u_1)+g(x,u_2)\in L^1(\Omega)$. 
By \cref{res:sgn}, we have that
\begin{equation*}
\int_\Omega F\sgn(v)\di\m\ge0,
\end{equation*}
which is equivalent to
\begin{align*}
\int_\Omega(g(x,u_1)-g(x,u_2))\,\sgn(u_1-u_2)\di\m
\le
\int_\Omega	(f_1-f_2)\,\sgn(u_1-u_2)\di\m
\end{align*}
and~\eqref{eq:oscillation} immediately follows. 
As a consequence, if $f_1=f_2$ then also $g(x,u_1)=g(x,u_2)$ and thus $F=0$ in~\eqref{eq:v_problem}. 
Therefore $\Delta_p v=0$ in $\Omega$ and thus
\begin{align*}
\sup_\Omega |v|
\le
C_{1,p}
\int_\Omega|\nabla v|^p\di\m
=
-C_{1,p}
\int_\Omega v\,\Delta_p v\di\m
=0
\end{align*}
by \cref{th:sobolev} and~\eqref{eq:def_mp_Lap}, so that $u_1=u_2$.
The proof is complete.
\end{proof}

\section{Applications}\label{proof3}

In this last section we briefly discuss some applications of our main results.

We begin by stating the following result, which shows that the Dirichlet problem in $W^{1,2}_0(\Omega)$ for the Laplcian operator with sufficiently well-behaved non-linearity admits a unique solution.

\begin{corollary}\label{res:bambi}
Let $G=(V,E)$ be a weighted locally finite graph. 
Let $\Omega\subset V$ be a bounded domain such that $\Omega^\circ\ne\varnothing$ and $\de\Omega\ne\varnothing$. 
Let $g\colon\Omega\times\R\to\R$ be a function such that $t\mapsto g(x,t)$ is $C^1$ and non-decreasing with $g(x,0)=\de_t g(x,0)=0$ for all $x\in\Omega$.
Let us set $\bar f(x)=g(x,0)$ for all $x\in\Omega$.
There exists $\delta>0$ with the following property: if $f\in L^2(\Omega)$ with $\|f-\bar f\|_{L^2(\Omega)}<\delta$, then the problem
\begin{equation*}
\begin{cases}
-\Delta u+g(x,u)=f 
& \text{in}\ \Omega^\circ\\[2mm]
u=0
& \text{on}\ \de\Omega
\end{cases}
\end{equation*}	
admits a unique solution $u\in W^{1,2}_0(\Omega)$.
\end{corollary}

The uniqueness part in \cref{res:bambi} is clearly immediately achieved by \cref{res:main_uniqueness}, while the existence part follows from the following result, which is inspired by the work~\cite{BC98}.

\begin{lemma}
Let $G=(V,E)$ be a weighted locally finite graph. 
Let $\Omega\subset V$ be a bounded domain such that $\Omega^\circ\ne\varnothing$ and $\de\Omega\ne\varnothing$. 
Let $g\colon\Omega\times\R\to\R$ be a function such that $t\mapsto g(x,t)$ is of class $C^1$ with $\de_t g(x,0)=0$ for all $x\in\Omega$.
Let us set $\bar f(x)=g(x,0)$ for all $x\in\Omega$. 
There exist $\delta,\eps>0$ with the following property: if $f\in L^2(\Omega)$ with $\|f-\bar f\|_{L^2(\Omega)}<\delta$, then the problem
\begin{equation*}
\begin{cases}
-\Delta u+g(x,u)=f 
& \text{in}\ \Omega^\circ\\[2mm]
u=0
& \text{on}\ \de\Omega
\end{cases}
\end{equation*}	
admits a unique solution $u\in W^{1,2}_0(\Omega)$ with $\|u\|_{W^{1,2}_0(\Omega)}<\eps$.
\end{lemma}

\begin{proof}
Let us consider the function $\mathcal F\colon W^{1,2}_0(\Omega)\to L^2(\Omega)$ defined by 
\begin{equation*}
\mathcal F(u)
=
-\Delta u+g(x,u)
\end{equation*}
for all $u\in W^{1,2}_0(\Omega)$.
Note that the map $\mathcal F$ is well defined, since $W^{1,2}_0(\Omega)\subset L^\infty(\Omega)$ with continuous embedding by \cref{th:sobolev} and thus also $x\mapsto g(x,u(x))\in L^\infty(\Omega)$ by the continuity property of~$g$ and by the fact that the number of vertices in~$\Omega$ is finite.
We additionally note that $\mathcal F\in C^1(W^{1,2}_0(\Omega),L^2(\Omega))$. 
Indeed, the Laplacian $\Delta$ is linear and the map $u\mapsto g(x,u)$ is of class $C^1$ thanks to the continuity properties of~$g$. 
Finally, we observe that the map $\mathcal F'(0)\colon W^{1,2}_0(\Omega)\to L^2(\Omega)$ is invertible, since $\mathcal F'(0)=-\Delta$ by the assumption that $\de_t g(x,0)=0$ for all $x\in\Omega$.
Since $\mathcal F(0)=\bar f$, the conclusion follows by the Inverse Function Theorem and the proof is complete.
\end{proof}

Our two main results Theorems~\ref{th:main} and~\ref{res:main_uniqueness} can be combined in order to achieve the well-posedness of a Yamabe-type problem on bounded domains, namely \cref{propcomb}.

\begin{proof}[Proof of \cref{propcomb}]
The function $g(x,t)=b(x)|t|^{q-1}t$, defined for $(x,t)\in\Omega\times\R$, satisfies the assumptions of \cref{res:main_uniqueness}, so that problem~\eqref{eq:yamabe_cor} admits at most one solution and we just need to deal with the existence issue.
If $\|a\|_{L^1(\Omega)}=0$, then clearly \mbox{$a=0$} and thus the null function $u=0$ is the unique solution of problem~\eqref{eq:yamabe_cor}.
If $\|a\|_{L^1(\Omega)}>0$ instead, then we apply \cref{th:main}.
Indeed, the function $f(x,t)=a(x)-b(x)|t|^{q-1}t$, defined for $(x,t)\in\Omega\times\R$, satisfies the assumptions of \cref{th:main}, and the conclusion thus follows in virtue of \cref{rem:Lambda}.
\end{proof}

We conclude our paper with the following uniqueness result for a Kazdan--Warner-type problem on bounded domains.
Its proof is a simple application of \cref{res:main_uniqueness} and is thus left to the reader.

\begin{corollary}
Let $G=(V,E)$ be a weighted locally finite graph. 
Let $\Omega\subset V$ be a bounded domain such that $\Omega^\circ\ne\varnothing$ and $\de\Omega\ne\varnothing$. 
Let $p\in[1,+\infty)$ and let $\alpha,\beta\in L^1(\Omega)$ be two non-negative functions.
For every $f\in L^1(\Omega)$ and $h\in L^1(\de\Omega)$, the Kazdan--Warner-type problem
\begin{equation}\label{eq:kazdan-warner}
\begin{cases}
-\Delta_p u + \alpha\,e^{\beta u} = f & \text{in}\ \Omega^\circ\\[2mm]
u = h & \text{on}\ \de\Omega
\end{cases}
\end{equation}	
admits at most one solution $u\in W^{1,p}(\Omega)$.
\end{corollary}



\end{document}